\documentclass[reqno]{amsart}
\usepackage{color}
\usepackage{amsmath}
\usepackage{amsfonts}
\usepackage{amssymb}
\usepackage{amsthm}
\usepackage{newlfont}
\usepackage{graphicx}

\newtheorem{thm}{Theorem}[section]
\newtheorem{cor}[thm]{Corollary}
\newtheorem{lem}[thm]{Lemma}
\newtheorem{prop}[thm]{Proposition}
\theoremstyle{definition}

\theoremstyle{remark}
\newtheorem{rem}[thm]{Remark}

\numberwithin{equation}{section}
\numberwithin{thm}{section}
\newtheorem*{rem*}{Remark}

\newcommand{\eps}{\varepsilon}

\newcommand{\lsm}{\lesssim}
\newcommand{\Z}{{\mathbb{Z}}}

\newcommand{\R}{{\mathbb{R}}}

\newcommand{\bs}[1]{\dot B_{\infty,\infty}^{#1}}
\newcommand{\ed}{\end {document}}

\newcounter{smalllist}

\title[Regularity upgrade of pressure]{A regularity upgrade of pressure}
\author[D. Li]{Dong Li}
\address{Department of Mathematics, The Hong Kong University of Science and Technology,
Clear Water Bay, Kowloon, Hong Kong}
\email{madli@ust.hk}
\author[X. Zhang]{Xiaoyi Zhang}
\address{Department of Mathematics, University of Iowa, 14 Maclean Hall, Iowa City, IA 52242}%
\email{xiaoyi-zhang@uiowa.edu}
\begin{document}

\begin{abstract}
For the incompressible Euler equations the pressure formally scales as a quadratic function of velocity. We provide several optimal regularity estimates on the pressure by using regularity of velocity
in various Sobolev, Besov and Hardy spaces. Our proof exploits the incompressibility condition in an essential way and is deeply connected with the classic Div-Curl lemma which we also generalise as a fractional Leibniz rule in Hardy spaces. To showcase the sharpness of results, we  construct a class of counterexamples at several end-points.
\end{abstract}
\maketitle
\section{Introduction}
The $n$-dimensional ($n\ge 2$) incompressible Euler equation takes the form
\begin{align} \label{g1}
\begin{cases}
\partial_t u + (u\cdot \nabla ) u = -\nabla p, \quad (t,x) \in (0,\infty) \times \mathbb R^n,\\
\nabla \cdot u=0, \\
u \Bigr|_{t=0}=u_0,
\end{cases}
\end{align}
where $u:\, \mathbb [0,\infty)\times \mathbb R^n \to \mathbb R^n$, $p:\, [0,\infty)\times \mathbb R^n \to \mathbb R$ represent velocity and pressure of the underlying fluid respectively. In this work we shall not consider the Cauchy problem or wellposedness issues at all. Instead we regard $u$ as a \emph{given} solution to \eqref{g1} in appropriate function spaces. Our main objective is to study
the regularity properties of the pressure $p$ in terms of the known velocity field $u$.
By taking the divergence on both sides of the first equation in \eqref{g1}, we get
\begin{align} \label{3}
-\Delta p = \nabla \cdot \bigl( ( u\cdot \nabla) u \bigr).
\end{align}
Equation \eqref{3} will be our main object of study.  To simplify the discussion we shall
completely ignore the explicit time dependence and only focus on spatial regularity. Regarding 
\eqref{3} as a Poisson problem for the pressure $p$, it is well known that $p$ is determined up to
a harmonic part. This degree of freedom can be eliminated by
supplying some decay conditions at spatial infinity. Alternatively to simplify matters, in the following
discussion, we shall take the convention that
$p$ is identified with the expression 
\begin{align*}
(-\Delta)^{-1} \nabla \nabla \bigl(  u\otimes u \bigr) = \sum_{l,k} (-\Delta)^{-1} \partial_l \partial_k (u_l u_k)
\end{align*}
which is well-defined
by putting suitable assumptions on the velocity vector $u$ (e.g. $u$ is Schwartz).  Our main focus is the quantitative estimates
of $p$ in terms of $u$ in various functional spaces. 

Note that since $u$ is divergence-free, one can rewrite \eqref{3} as
\begin{align} \label{4}
-\Delta p = \sum_{k,l} \partial_k \partial_l ( u_k u_l).
\end{align}
Alternatively, we can rewrite \eqref{4} as
\begin{align} \label{4a}
-\Delta p= \sum_{k,l} (\partial_k u_l ) (\partial_l u_k).
\end{align}

There are some subtle differences between the expressions
\eqref{3}--\eqref{4a}. Assume that $u \in W^{1,q}(\R^n)$ for some $n<q<\infty$. If we only make use of \eqref{3}, then
\begin{align*}
\nabla p= (-\Delta)^{-1} \nabla \nabla \cdot (u\cdot \nabla u),
\end{align*}
which gives $\nabla p \in L^q(\R^n)$.  By using \eqref{4} it is easy to check that 
$p\in L^q(\mathbb R^n)$. Thus we obtain $p \in W^{1,q}(\mathbb R^n)$.
On the other hand, if we use \eqref{4a}, then clearly $p \in W^{2, \frac q2} (\mathbb R^n)$. This
is apparently a better estimate in view of the embedding $W^{2,\frac q2} (\mathbb R^n) \hookrightarrow W^{1,q}(\mathbb R^n)$. As it turns out,
this ``upgrade of regularity" phenomenon is quite generic. For example it can be generalised to Sobolev spaces $W^{s,q}$ with $0<s\le 1$,
$2<q<\infty$ and even Besov spaces.
We have the following theorem.

\begin{thm}[Sobolev and Besov] \label{t0}
Suppose $u\in \mathcal S(\mathbb R^n)$. Then 
for $0\le s \le 1$ and $1<q<\infty$, we have
\begin{align}
\| p \|_{W^{2s,q} (\mathbb R^n)} \lesssim  \| u \|_{W^{s,2q}(\mathbb R^n)}^2.\label{sob}
\end{align}
If $1\le q, r\le \infty$ and 
$0<s<1$, then
\begin{align}
\|p\|_{\dot B^{2s}_{q,r}(\mathbb R^n)}\lsm \|u\|_{\dot B^s_{2q,2r} (\mathbb R^n)}^2. \label{besov}
\end{align}
In particular for H\"older spaces we have for $0<s<1$,
\begin{align}
\| p \|_{\dot B^{2s}_{\infty,\infty} } \lesssim  \| u
\|_{\dot B^s_{\infty,\infty}}^2.\label{bes}
\end{align}
On the other hand, 
for $s=1$ the corresponding Besov estimate does not hold and can be replaced by
\begin{align} \label{bes_stronger}
\| p\|_{\dot B^2_{\infty,\infty} } \lesssim  \| \nabla u  \|_{\infty}^2.
\end{align}
\end{thm}
\begin{rem*}
\eqref{bes_stronger} is a simple consequence of  \eqref{4a} and thus we omit the proof.
\end{rem*}
\begin{rem*}
For Schwartz functions $f$: $\mathbb R^n \to \mathbb R^n$, $g$: $\mathbb R^n \to \mathbb R^n$ with the property
$\nabla \cdot f= \nabla \cdot g =0$,  one can consider
the bilinear operator
\begin{align*}
B(f,g)= \sum_{l,k=1}^n \partial_l \partial_k \Delta^{-1} ( f_l g_k).
\end{align*}
Same proof as in Theorem \ref{t0} yields that for $0\le s\le 1$ and $1<q<\infty$,
\begin{align*}
\| B(f,g) \|_{W^{2s,q}} \lesssim \| f \|_{W^{s,2q}} \| g \|_{W^{s,2q}};
\end{align*}
and for $0<s<1$, $1\le q,r\le \infty$,
\begin{align*}
\| B(f,g) \|_{\dot B^{2s}_{q,r}} \lesssim \| f \|_{\dot B^{s}_{2q,2r}} \| g  \|_{\dot B^s_{2q,2r}}.
\end{align*}

\end{rem*}
\begin{rem*}
The estimate \eqref{bes} shows that for $0<s<\frac 12$,
\begin{align*}
\| p \|_{\dot C^{2s}} \lesssim \| u \|_{\dot C^s}^2;
\end{align*}
and for $\frac 12 <s<1$,
\begin{align*}
\| \nabla p \|_{\dot C^{2s-1}} \lesssim \| u \|_{\dot C^s}^2.
\end{align*}
One should recall that $\|f \|_{\dot C^{s}} \sim \|f \|_{\dot B^s_{\infty,\infty}}$ for $0<s<1$. Thus
the Besov formulation connects these two estimates in a most natural way.
\end{rem*}

Note from \eqref{4} it is evident that $p=\mathcal R( O(u^2))$ where
$\mathcal R$ is a Riesz-type operator. The usual product rule in
H\"older spaces says that if $f,\,g \in C^{s}$, then $fg \in C^{s}$.
Thus by \eqref{4} one should only expect $p\in C^{s}$ if $u \in
C^{s}$. However here by using Theorem \ref{t0}, one can prove that $p \in \dot
B^{2s}_{\infty,\infty}$ as long as $0<s<1$. Roughly speaking, we are
asserting that the $2s$-derivative on $p$ can fall ``evenly'' into
each composing velocity:
\begin{align*}
|\nabla|^{2s} p \sim \mathcal R( |\nabla|^{s} u |\nabla|^{s} u),
\end{align*}
and there do not appear terms such as $O(|\nabla|^{s-} u |\nabla|^{s+} u)$.

There exists an analogue of Theorem \ref{t0} in Hardy space. The following theorem can be regarded
as the case $q=2$ in Theorem \ref{t0}.

\begin{thm}[Hardy space]\label{t2}
Let $0<s\le 1$. Then for $u\in \mathcal S(\mathbb R^n)$, 
\begin{align}\label{40}
\||\nabla|^{2s}p\|_{\mathcal H^1(\mathbb R^n)}\lsm \|u\|_{W^{s,2}(\mathbb R^n)}^2.
\end{align}
In the case $s=1$, the operator can be replaced by a general second order derivative $\partial^2=\partial_i \partial_j$ for any $i,j \in \{1,\cdots, n\}$.
\end{thm}
\begin{rem*}
For $0<s<1$ one has the stronger estimate $ p \in \dot B^{2s}_{1,1}$ thanks to theorem \ref{t0}.
\end{rem*}

 In the proof of Theorem \ref{t2} we need to exploit the incompressibility of velocity which provides cancelation
 of  some
 high frequency interaction terms in the nonlinearity.  As it turns out,  in the 3D case,  this is deeply linked to
 the ``Div-Curl" lemma in Coifman-Lions-Meyer-Semmes \cite{clms}. In its simplest formulation,
 the Div-Curl lemma asserts that if $\nabla \cdot f= \nabla \times g=0$, then
 \begin{align*}
\|f\cdot g\|_{\mathcal H^1}\lsm \|f\|_p \|g\|_{p'},
\end{align*}
 where $1<p<\infty$ and $p^{\prime}=p/(p-1)$.  In light of the proof in Theorem \ref{t2}, we can obtain the
following generalisation which can also be regarded as a fractional Leibniz rule in Hardy space.
\begin{thm}[Generalised Div-Curl lemma]\label{dc}
Let $1<p_1,p_2<\infty$ and $-1<s<\infty$. Then for any Schwartz $f$, $g$: $\mathbb R^3\to \mathbb R^3$ with
\begin{align*}
\nabla\cdot f=0,\ \ \nabla\times g=0,
\end{align*}
we have
\begin{align*}
\||\nabla|^s(f\cdot g)\|_{\mathcal H^1}\lsm \||\nabla|^s f\|_{p_1}\|g\|_{p_1'}+\|f\|_{p_2'}\||\nabla|^s g\|_{p_2}
\end{align*}
whenever the RHS is finite. When $-1<s\le 0$, we have 
\begin{align*}
\||\nabla|^s(f\cdot g)\|_{\mathcal H^1}\lsm 
\min\{ \||\nabla|^s f\|_{p_1}\|g\|_{p_1'}, \;\, \|f\|_{p_2'}\||\nabla|^s g\|_{p_2} \}.
\end{align*}
\end{thm}
\begin{rem*}
For $s>0$ one does not need the ``Div-Curl" condition to derive the estimate. The only nontrivial case is $-1<s\le 0$. 
\end{rem*}

In Theorem \ref{t0}, for the Besov case we only considered the regime $0<s<1$ and left out the cases $s=0$ and $s=1$. It is instructive
to investigate these end-point cases. As it turns out, the corresponding regularity estimate for
pressure fails in general. To clarify this point we construct counterexamples.

\begin{thm}\label{prop1}
The estimate \eqref{bes} fails in the case $s=0,1$. More specifically, for any $\eps>0$, there exists divergence free $u\in \mathcal S(\R^n)$ for with
\begin{align*}
\|u\|_{\dot B^{1}_{\infty,\infty} }+\|u\|_{L^2}\le 1,
\end{align*}
but
\begin{align*}
\|p\|_{\dot B^{2}_{\infty,\infty} }>\frac 1{\eps}.
\end{align*}
A similar statement holds for $s=0$. Here $\mathcal S(\mathbb R^n)$ is the class of Schwartz functions. 
\end{thm}
\begin{rem*}
Our construction shows that the inflation occurs at high frequencies. 
\end{rem*}
In the estimate \eqref{bes}, when $s=\frac 12$, $p\in \dot
B^{1}_{\infty,\infty} $. The norm of $\dot B^{1}_{\infty,\infty} $ is much weaker than the norm of
$C^1$. Of course it is not difficult to construct a function which lies in $\dot B^{1}_{\infty,\infty} $
but not in $C^1$. However, it is not obvious  to construct a
$p\notin C^1$ starting from a divergence free $u \in C^{\frac 12}$. The following theorem gives such
 a result.

\begin{thm}\label{ct2}
There exists $u\in C^{\frac 12} \cap L^2$ for which $p
\notin C^1$.
\end{thm}

The last result says the estimate \eqref{40} fails at the endpoint case $s=0$.

\begin{prop}\label{prop5}
There exists divergence-free $u \in \mathcal S(\mathbb R^n)$ such
that $p \notin L^1(\mathbb R^n)$.
\end{prop}

We should remark that the case of domains with appropriate boundary conditions can be explored further and we plan to address it elsewhere. 

The rest of this paper is organised as follows. In Section 2 we introduce some basic notation and collect
some preliminary estimates. In Section 3 we give the regularity estimates of pressure in aforementioned function
spaces. In Section 4 we give the construction of counterexamples at various end-point cases.

\subsection*{Acknowledgements}
D. Li was supported in part by a start-up grant from HKUST and HK RGC grant 16307317. X. Zhang was supported by Simons Collaboration grant.

\section{preliminaries}
For any real number $a\in \mathbb R$, we denote by $a+$ the quantity $a+\epsilon$ for sufficiently small
$\epsilon>0$. The numerical value of $\epsilon$ is unimportant and the needed
smallness of $\epsilon$ is usually clear from the context. 
The notation $a-$ is similarly defined. 

For any two quantities $X$ and $Y$, we denote $X \lesssim Y$ if
$X \le C Y$ for some constant $C>0$. Similarly $X \gtrsim Y$ if $X
\ge CY$ for some $C>0$. We denote $X \sim Y$ if $X\lesssim Y$ and $Y
\lesssim X$. The dependence of the constant $C$ on
other parameters or constants are usually clear from the context and
we will often suppress  this dependence. We shall denote
$X \lesssim_{Z_1, Z_2,\cdots,Z_k} Y$
if $X \le CY$ and the constant $C$ depends on the quantities $Z_1,\cdots, Z_k$.

We denote by $\mathcal S(\mathbb R^n)$ the space of Schwartz functions and $\mathcal
S^{\prime}(\mathbb R^n)$ the space of tempered distributions. 
  For any function $f:\; \mathbb R^n\to
\mathbb R$, we use $\|f\|_{L^q(\mathbb R^n)}$, $\|f\|_{L^q}$ or sometimes $\|f\|_q$ to denote
the  usual Lebesgue $L^q$ norm  for $0< q \le
\infty$. For $s>0$, $1<q<\infty$, we recall the Sobolev norms
\begin{align*}
&\|f \|_{\dot W^{s,q}(\mathbb R^n)} = \| |\nabla|^s f \|_{q},
\quad \| f\|_{W^{s,q}(\mathbb R^n)}= \| f\|_q+ \| |\nabla|^s f \|_q.
\end{align*}

For a sequence of real numbers $(a_j)_{j={-\infty}}^{\infty}$, we denote
\begin{align*}
(a_j)_{l_j^q}=\|(a_j)_{j\in \mathbb Z} \|_{l^q} =\begin{cases}
(\sum_{j\in \mathbb Z} |a_j|^q )^{\frac 1q}, \qquad \text{if $0<q<\infty$}, \\
\sup_{j} |a_j|, \qquad \text{if $q=\infty$}.
\end{cases}
\end{align*}
We shall often use mixed-norm notation. For example, for a sequence of functions $f_j: \;
\mathbb R^n \to \mathbb R$,
we will denote (below $0<r<\infty$)
\begin{align*}
\| (f_j)_{l_j^r} \|_{q} =\| (\sum_j |f_j(x)|^r)^{\frac 1r} \|_{L^q_x(\mathbb R^n)},
\end{align*}
with obvious modification for $q=\infty$.

We use the following convention for the Fourier transform: 
\begin{align*}
(\mathcal F f)(\xi)=\hat f(\xi)=\int_{\mathbb R^n}e^{- ix\cdot \xi}f(x) dx. 
\end{align*}
$\mathcal F^{-1}$ is the inverse Fourier transform:
\begin{align*}
\mathcal F^{-1} g(x)= \frac  1 {(2\pi)^n} \int_{\mathbb R^n} e^{ ix\cdot \xi}g(\xi) d\xi. 
\end{align*}
For $s\in \mathbb R$, the fractional Laplacian $|\nabla|^s= (-\Delta)^{s/2}$ corresponds to the 
multiplier $|\xi|^s$ on the Fourier side (whenever it is well-defined). Sometimes we also
denote  $|\nabla|^s$ as $D^s$.

We first introduce the Littlewood-Paley operators.

Let $\tilde \phi\in C_c^\infty(\R^n)$ be such that
\begin{align*}
\tilde\phi(\xi )=\begin{cases} 1, & |\xi|\le 1\\
0, & |\xi|>\frac 7 6.
\end{cases}
\end{align*}
Let $\phi_c (\xi)=\tilde \phi(\xi)-\tilde \phi(2\xi)$ which is supported on $\frac 12\le |\xi| \le \frac 76$. 
For any $f \in \mathcal S(\mathbb R^n)$, $j \in \mathbb Z$, define
\begin{align*}
 &\mathcal F ( P_{\le j} f )(\xi)=\widehat{P_{\le j} f} (\xi) = \tilde \phi(2^{-j} \xi) \hat f(\xi), \\
 &\widehat{P_j f} (\xi) = \phi_c(2^{-j} \xi) \hat f(\xi), \qquad \xi \in \mathbb R^n.
\end{align*}
We will denote $P_{>j} = I-P_{\le j}$ ($I$ is the identity operator).  Sometimes for simplicity of
notation (and when there is no obvious confusion) we will write $f_j = P_j f$, $f_{\le j} = P_{\le j} f$ and
$f_{a\le\cdot\le b} = \sum_{a\le j\le b} f_j$.
By using the support property of $\phi$, we have $P_j P_{j^{\prime}} =0$ whenever $|j-j^{\prime}|>1$.
This property will be useful in product decompositions. For example the Bony paraproduct for a pair of functions
$f,g$ take the form
\begin{align*}
f g = \sum_{i \in \mathbb Z} f_i \tilde g_i + \sum_{i \in \mathbb Z} f_i g_{\le i-2} + \sum_{i \in \mathbb Z}
g_i f_{\le i-2},
\end{align*}
where $\tilde g_i = g_{i-1} +g_i + g_{i+1}$.

The fattened operators $\tilde P_j$ are defined by
\begin{align*}
 \tilde{P}_j = \sum_{l=-{n_1}}^{n_2} P_{j+l},
\end{align*}
where $n_1\ge 0$, $n_2\ge 0$ are some finite integers whose values play no role in the argument.

We will often use the following Bernstein inequalities without explicit mentioning.

\begin{prop}[Bernstein inequality]
Let $1\le p\le q\le \infty$. For any $f\in L^p(\mathbb R^n)$, $j \in \mathbb Z$, we have
\begin{align*}
&\|P_{\le j} f\|_q +\|P_jf \|_q\lesssim 2^{jn(\frac 1p-\frac 1q)}\|f\|_p; \\
&\||\nabla|^s P_{\le j } f\|_p \lsm 2^{js}\|f\|_p,\quad \forall\, s\ge 0;\\
&\|P_j f\|_p \sim 2^{-js} \||\nabla|^s P_j f\|_p, \quad \forall\, s\in \mathbb R;\\
&\|P_{>j} f\|_p \lesssim 2^{-js} \||\nabla|^s f\|_p, \quad \forall\, s\ge 0. 
\end{align*}
\end{prop}

For $s\in \mathbb R$, $1\le p, q \le \infty$, the homogeneous Besov
$\dot B^s_{p,q}$ (semi)-norm is given by
\begin{align*}
\|f \|_{\dot B^s_{p,q}(\mathbb R^n)}
= \|(2^{js} \|f_j\|_p )\|_{l_j^q}
=\begin{cases}
(\sum_{j \in \mathbb Z} 2^{jqs} \| P_j f\|_{L_x^p(\mathbb R^n)}^q )^{\frac 1q},
\quad \text{if $1\le q<\infty$};\\
\sup_{j\in \mathbb Z} 2^{js} \| P_j f\|_{L_x^p(\mathbb R^n)},\quad
\text{if $q=\infty$}.
\end{cases}
\end{align*}
For $s\in \mathbb R$, $1\le p,\, q\le \infty$, 
\begin{align*}
\| f \|_{B^s_{p,q}(\mathbb R^n)} = \| P_{\le 0} f \|_p + \| (2^{js} \| f_j \|_p) \|_{l_j^q(j\ge 1)}.
\end{align*}
For $s>0$ and $1\le p, q\le \infty$, it is easy to check that
\begin{align*}
\| f \|_{B^s_{p,q}} \sim \| f \|_p +  \|f \|_{\dot B^s_{p,q}}.
\end{align*}

\begin{prop}[Continuity in ``$s$"]
Let $1\le p, q\le \infty$. Suppose $ \|f \|_{B^s_{p,q}} <\infty$, then
\begin{align*}
\lim_{\substack{ \tilde s \to s \\ \tilde s <s}} \| f \|_{B^{\tilde s}_{p,q}} = \| f \|_{B^s_{p,q}}.
\end{align*}
On the other hand, if for some $\epsilon_0>0$, 
\begin{align*}
\sup_{\tilde s \in (s-\epsilon_0, s)} \| f \|_{B^{\tilde s}_{p,q}} <\infty,
\end{align*}
then $\| f \|_{B^s_{p,q}} <\infty$, and
\begin{align*}
\lim_{\substack{ \tilde s \to s \\ s-\epsilon_0<\tilde s <s}} \| f \|_{B^{\tilde s}_{p,q}} = \| f \|_{B^s_{p,q}}.
\end{align*}
If for some $\epsilon_1>0$, $\|f \|_{B^{s+\epsilon_1}_{p,q}} <\infty$, then
\begin{align*}
\lim_{\substack{\tilde s\to s\\ \tilde s>s}} \|f \|_{B^{\tilde s}_{p,q}} = \| f \|_{B^s_{p,q}}.
\end{align*}
\end{prop}
\begin{rem*}
This simple proposition explains the folklore fact that the limit $\alpha \to 1$ for H\"older $C^{\alpha}$, $0<\alpha<1$, is $B^1_{\infty,\infty}$;
and the limit $\alpha \to 0$ for H\"older $C^{1,\alpha}$, $0<\alpha<1$ is $B^1_{\infty,\infty}$.
\end{rem*}
\begin{proof}
If $q<\infty$, one can just use Monotone Convergence to obtain
\begin{align*}
\lim_{\substack{ \tilde s \to s \\ \tilde s <s}} (2^{j\tilde s} \| P_j f \|_p)_{l_j^q(j\ge 1)}
= (2^{js} \| P_j f \|_p)_{l_j^q(j\ge 1)}.
\end{align*}
Also for $q<\infty$, by using Dominated convergence, we have
\begin{align*}
\lim_{\substack{ \tilde s \to s \\ s<\tilde s<s+\epsilon_1}} (2^{j\tilde s} \| P_j f \|_p)_{l_j^q(j\ge 1)}
= (2^{js} \| P_j f \|_p)_{l_j^q(j\ge 1)}.
\end{align*}
If $q=\infty$, one observes that
\begin{align*}
&\sup_{\tilde s<s} \sup_{j\ge 1} 2^{j\tilde s } \| P_jf \|_p = \sup_{j\ge 1} 2^{js} \| P_jf \|_p;\\
&\lim_{\substack{\tilde s \to s \\ s<\tilde s<s+\epsilon_1}} \sup_{j\ge 1} 2^{j\tilde s} \|P_j f\|_p
= \sup_{j\ge 1} 2^{js} \| P_j f \|_p.
\end{align*}

\end{proof}

We will use the following characterisation for the $\dot C^{\alpha}$ space. Recall that for
$0<\alpha<1$,
\begin{align*}
\|f \|_{\dot C^{\alpha}} := \sup_{x\ne y} 
\frac{|f(x)-f(y)|} { |x-y|^{\alpha}},
\end{align*}
and
\begin{align*}
\|f \|_{C^{\alpha}} = \|f\|_{\infty} + \|f \|_{\dot C^{\alpha}}.
\end{align*}

\begin{lem}  \label{lem2.2}
Let $0<\alpha<1$. Then for any $f$ with $\|f \|_{ \dot C^{\alpha}(\mathbb R^n)} <\infty$, we have
\begin{align*}
\|f\|_{\dot C^{\alpha}}& \sim \|f\|_{\dot B^{\alpha}_{\infty,\infty}}.
\end{align*}
\end{lem}
\begin{rem*}
Similarly for $\alpha=1$ and $f \in \mathcal S(\mathbb R^n)$, one has the equivalence
\begin{align*}
\sup_{h\ne 0, x\in \mathbb R^n}
\frac{|f(x+h)+f(x-h)-2f(x)|} {|h|} \sim \|f \|_{\dot B^1_{\infty,\infty}}.
\end{align*}
\end{rem*}
\begin{rem*}
It follows easily that for $0<\alpha<1$, and any $f\in \mathcal S(\mathbb R^n)$,  we have
\begin{align*}
\| \nabla f \|_{\dot C^{\alpha}} \sim \| \nabla f \|_{\dot B^{\alpha}_{\infty,\infty}}
\sim \| f \|_{\dot B^{1+\alpha}_{\infty,\infty}}.
\end{align*}
\end{rem*}
\begin{proof}
Easy to check that $\|f \|_{\dot B^{\alpha}_{\infty,\infty}} \lesssim \| f\|_{\dot C^{\alpha}}$. To show the reverse inequality,
one can fix $x\ne y$ and note that with the assumption $\|f\|_{\dot C^{\alpha}}<\infty$ one has
\begin{align*}
\lim_{J\to \infty}|(P_{\le -J} f)(x) - (P_{\le -J} f)(y)| \lesssim \lim_{J\to \infty}  \| \nabla P_{\le -J} f \|_{\infty} 
\cdot |x-y| =0.
\end{align*}
This implies the point-wise identity:
\begin{align*}
f(x)- f(y) = \sum_{j\in \mathbb Z} ( (P_j f)(x) - (P_j f)(y) ).
\end{align*}
 Thus
\begin{align*}
|f(x) -f(y)| &\le \sum_{j\le J_0} \| \nabla P_j f \|_{\infty} \cdot |x-y| + \sum_{j>J_0} 2\cdot \|P_j f\|_{\infty} \notag \\
& \lesssim 2^{J_0(1-\alpha)} |x-y| \| f \|_{\dot B^{\alpha}_{\infty,\infty}} + 2^{-J_0\alpha} \| f \|_{\dot B^{\alpha}_{\infty,\infty}}.
\end{align*}
Choosing $2^{J_0} \sim |x-y|^{-1}$ then yields the result.
\end{proof}

\begin{lem}
Let $0<\alpha<1$.
Suppose $f \in \mathcal S^{\prime}(\mathbb R^n)$ satisfies
\begin{align*}
\|f \|_{\dot B^{\alpha}_{\infty,\infty}} = \sup_{j}  ( \| P_j f \|_{\infty} \cdot 2^{j\alpha} ) <\infty.
\end{align*}
Then
\begin{align*}
f= \tilde f +p,
\end{align*}
where $p$ is a polynomial, and $\tilde f$ is a continuous function satisfying
\begin{align*}
\| \tilde f \|_{\dot C^{\alpha}} \sim \| f \|_{\dot B^{\alpha}_{\infty,\infty}}.
\end{align*}
\end{lem}

\begin{proof}
Denote $f_J= \sum_{-J \le j \le J} P_j f$. Clearly $\|f_J\|_{\dot C^{\alpha}} \lesssim \|f \|_{\dot B^{\alpha}_{\infty,\infty}}$. Set
$$\tilde f_J(x)= f_J(x)- f_J(0).$$
Note that $|\tilde f_J(x) | \lesssim 1+|x|^{\alpha}$.  By Arzel\`a-Ascoli one can extract a subsequence $\tilde f_{J_k} $ converging
locally uniformly to a limit $\tilde f$.  Clearly $\| \tilde f\|_{\dot C^{\alpha}} \lesssim \|f \|_{\dot B^{\alpha}_{\infty,\infty}}$. By the previous
lemma we have $\| \tilde f \|_{\dot C^{\alpha}} \sim \| \tilde f \|_{\dot B^{\alpha}_{\infty,\infty}}$.
 Fix any $j_0 \in \mathbb Z$. Clearly if $J$ is large we have $P_{j_0} 
\tilde f_J= P_{j_0} f$. By Lebesgue Dominated Convergence and sending $J_k$ to infinity we obtain $P_{j_0} \tilde f = P_{j_0} f$
for any $j_0 \in \mathbb Z$.  Thus $\tilde f$ and $f$ differ at most by a polynomial.

\end{proof}

The following corollary records the usual fact that for $0<\alpha<1$ and $f\in \mathcal S^{\prime}(\mathbb R^n)$,
one has the equivalence $ \| f \|_{B^{\alpha}_{\infty,\infty}} \sim \| f \|_{C^{\alpha}}$.
\begin{cor}
Let $0<\alpha<1$. 
Suppose $f$ is continuous and bounded on $\mathbb R^n$. Then
\begin{align*}
\| f \|_{\dot B^{\alpha}_{\infty,\infty}} \sim \| f\|_{\dot C^{\alpha}}.
\end{align*}
\end{cor}
\begin{proof}
If $\|f \|_{\dot C^{\alpha}} <\infty$, clearly then $\| f\|_{\dot B^{\alpha}_{\infty,\infty}} \lesssim \| f \|_{\dot C^{\alpha}}$. Conversely if
$\| f \|_{\dot B^{\alpha}_{\infty,\infty}}<\infty$, then by the previous lemma, we have $f =\tilde f+ p$, where $p$ is a polynomial, and
$\| \tilde f\|_{\dot C^{\alpha}} \sim \| f\|_{\dot B^{\alpha}_{\infty,\infty}}$. Since $f$ is continuous and bounded,  by considering the
limit of the ratios $f(x)/|x|$ and $\tilde f (x)/|x|$ as $|x|\to \infty$, it follows easily that
$p$ must be a constant.  Thus $f = \tilde f+ \operatorname{const}$ from which the desired result follows.
\end{proof}
\begin{rem*}
Alternatively one can use the condition $\| f \|_{\infty}<\infty$ to obtain that for any $x\ne y$, 
\begin{align*}
\lim_{J\to \infty}  |P_{\le -J} f (x)
-P_{\le -J} f (y)|  &\lesssim \lim_{J\to \infty}  \| \nabla P_{\le -J} f\|_{\infty} |x-y|  \notag \\
&\lesssim \lim_{J\to \infty} 2^{-J} \| f \|_{\infty} |x-y| =0.
\end{align*}
Repeating the proof in Lemma \ref{lem2.2} then yields easily $\|f \|_{\dot C^{\alpha}} \lesssim \|f \|_{\dot B^{\alpha}_{\infty,\infty}}$.
\end{rem*}

We also need the following maximal function estimates.

\begin{lem}[Fefferman-Stein inequality, \cite{FS}] \label{lem_FS}
 Let $f=(f_j)_{j=1}^{\infty}$ be a sequence of functions in $\mathbb R^n$. Let $1<q,r<\infty$. There exists
 a positive constant $A_{r,q,n}>0$ such that
 \begin{align*}
  \biggl\| \bigl( \sum_{j=1}^{\infty}  |\mathcal M_1 f_j|^r \bigr)^{\frac 1r} \biggr\|_q \le A_{r,q,n}
  \biggl\| \bigl( \sum_{j=1}^{\infty}  | f_j|^r \bigr)^{\frac 1r} \biggr\|_q.
 \end{align*}
Here $\mathcal M_1 g$ is the maximal function defined by (here $|Q|$ denotes $n$-dimensional Lebesgue measure of $Q$):
$
 (\mathcal M_1 g)(x) = \sup_{Q} \frac 1 {|Q|}\int_{Q} |g(y)|dy,
$
where the $\operatorname{sup}$ is taken over all cubes $Q$ centered at $x$.
\end{lem}
\begin{rem}
Of course  the maximal operator $\mathcal M_1$ can be replaced by the usual uncentered maximal operator $\mathcal M$
on balls:
\begin{align*}
 (\mathcal M g)(x) = \sup_{B} \frac 1 {|B|} \int_{B} |g(y)|dy,
\end{align*}
where the $\operatorname{sup}$ is taken over all balls $B$ containing the point $x$.
\end{rem}

As an immediate consequence of Lemma \ref{lem_FS}, we have the following estimate.

\begin{cor}
Let $\tilde P_{j}$ be a family of Littlewood-Paley type projection operators adapted to frequency 
$|\xi|\sim 2^j$,
and obey the bound (often this is easy to check)
\begin{align*}
\sup_{j\in {\Z}} |(\tilde P_{j} g )(x)| \le \operatorname{const}\cdot (\mathcal M g)(x), \qquad \forall\, g.
\end{align*}
Then we have
\begin{align}
  \biggl\| \bigl( \sum_{j}  |\tilde P_{j} f_j|^r \bigr)^{\frac 1r} \biggr\|_q \lesssim
  \biggl\| \bigl( \sum_{j}  | f_j|^r \bigr)^{\frac 1r} \biggr\|_q,\quad \forall \ 1<q,r<\infty.\label{cor:fs}
 \end{align}
 \end{cor}

Several acceptable choices of $\tilde P_j$ include, for example, the fattened projection: $\tilde P_j=P_{j-1}+P_j+P_{j+1}$ and the projections involving derivatives: $\tilde P_j=2^{-js} |\nabla|^s P_j$, $s>0$. In later sections, we shall often use this inequality (mostly in the
case $r=2$) without explicit mentioning. Note that
alternatively one can prove \eqref{cor:fs} by using the usual
vector-valued CZ operator theory.

Note that in \eqref{cor:fs}, $q\in (1,\infty)$. Next we shall discuss an analogue of this estimate when $q\le 1$. Such estimate will be used in dealing with estimates in Hardy spaces. And we start by a short review of Hardy spaces.\footnote{To keep it simple, we only discuss the estimates in Hardy space $\mathcal H^1$, 
the extension to general Hardy space $\mathcal H^p$ can be done without much effort.}

 We recall the definition of Hardy space $\mathcal
H^1$:
\begin{align*}
\mathcal H^1 (\mathbb R^n)= \Bigl\{ f \in L^1(\mathbb R^n):\;\;
\mathcal R_j f \in L^1, \quad \forall\, j=1,\cdots, n \Bigr\},
\end{align*}
where $\mathcal R_j$ are the Riesz transforms ($\hat R_j(\xi) = -i
\xi_j /|\xi|$) on $\mathbb R^n$ (for $n=1$ it is just the Hilbert
transform) with the norm
\begin{align*}
\| f \|_{\mathcal H^1} = \|f\|_1 + \sum_{j=1}^n \| \mathcal R_j f
\|_1.
\end{align*}
We shall not need the atomic or maximal function characterization of
$\mathcal H^1$. On the other hand, it is well-known that (see e.g.
Chapter 6.4 on p37 of \cite{Grafakos2})
\begin{align*}
\|f \|_{\mathcal H^1} \sim \| \mathcal S f \|_1,
\end{align*}
where $\mathcal S f$ is the usual Littlewood-Paley square function
$\mathcal S f = (\sum_{j\in \mathbb Z} |P_{j} f |^2)^{\frac 12}$.
In view of this equivalence, we shall conveniently use the
definition
\begin{align} \label{def_Hardy}
\|f \|_{\mathcal H^1} = \| \mathcal S f \|_1
\end{align}
in this note without explicit mentioning.

\begin{lem} \label{max_1t}
Suppose $u\in \mathcal S(\mathbb R^n)$ with $\operatorname{supp}(\hat u) \subset \{\xi:\; |\xi|<t\}$ for some $t>0$.
Then for any $0<r<\infty$,
\begin{align*}
\sup_{z \in \mathbb R^n} \frac{|u(x-z)|}{(1+t|z|)^{\frac nr}}
\lesssim_{r,n} (\mathcal M(|u|^r)(x) )^{\frac 1r}.
\end{align*}
\end{lem}

\begin{rem*}
To understand the meaning of the lemma one can take $x=0$ and $t=1$. Thanks to
 frequency localization, for given $z$
the value $u(z)$ is mainly controlled by the $L^r$ norm of $|u|$ in a ball of size $O(|z|)$ centered at the origin.
The factor $|z|^{\frac nr}$ then provides the needed normalization for the maximal function.
\end{rem*}

\begin{rem*}
See Lemma 2.3 of \cite{L18} for a proof under a much weaker condition that $u$ is a tempered
distribution (i.e. one can replace $\mathcal S(\mathbb R^n)$ by 
$\mathcal S^{\prime} (\mathbb R^n)$).
\end{rem*}

\begin{proof}
By scaling we may assume $t=1$. By using translation we only need to consider the case $x=0$.
Assume first $1\le r<\infty$.  Since $u=P_{<2} u$, we have
\begin{align*}
u(z) = \int \psi(z-y) u(y)dy,
\end{align*}
where $\psi$ is the Schwartz function corresponding to $P_{<2}$. Clearly then
\begin{align*}
|u(z)| &\le \int |\psi(z-y)|^{\frac 1r} |u(y)| |\psi(z-y)|^{1-\frac 1r} dy \notag \\
& \lesssim \bigl( \int |\psi(z-y)| |u(y)|^r dy \bigr)^{\frac 1r} \notag \\
& \lesssim (1+|z|)^{\frac nr} ( \mathcal M (|u|^r)(0) )^{\frac 1r},
\end{align*}
where in the last step we used the simple inequality
\begin{align*}
\int |\psi(z-y)| |f(y)| dy \lesssim (1+|z|)^n (\mathcal M f)(0).
\end{align*}

Now consider $0<r<1$. Clearly
\begin{align*}
|u(z)| & \le \int |\psi(z-y)| |u(y)|^{r} |u(y)|^{1-r} dy \notag \\
& \le \int |\psi(z-y)| |u(y)|^r (1+|y|)^{\frac n r (1-r)} dy
\sup_{x} \frac {|u(x)|^{1-r} }{(1+|x|)^{\frac nr (1-r)}} \notag \\
& \lesssim (1+|z|)^{\frac nr} \mathcal M (|u|^r)(0) \cdot \sup_{x}
\frac {|u(x)|^{1-r} }{(1+|x|)^{\frac nr (1-r)}}.
\end{align*}
This then easily leads to the desired inequality.

\end{proof}

\begin{cor} \label{cor_fmr}
Let $0<r<\infty$ and $c_1>0$. Then for any $j \in \mathbb Z$ and Schwartz $f$ with $\operatorname{supp}(\hat f) \subset \{\xi:\, |\xi|
< c_1 2^j\}$, we have
\begin{align*}
|(P_{\le j} f)(x) | \lesssim_{r,c_1,n} (\mathcal M (|f|^r) (x) )^{\frac 1r}.
\end{align*}
In the above estimate the operator $P_{\le j}$ can be replaced by $P_j$ or any $\tilde P_{\le j}$, $\tilde P_j$.
\end{cor}
\begin{proof}
The case $1\le r<\infty$ is trivial. In particular it follows from the estimate
$|f(x)|\lesssim \mathcal M f(x) \lesssim_r ((\mathcal M |f|^r)(x) )^{\frac 1r}$ when $r>1$.
Now consider $0<r<1$ and without loss of generality assume $j=1$. Clearly 
\begin{align*}
|(P_{\le 1} f)(x) | & \le \int |\psi(y)| |f(x-y) | dy \notag \\
& \lesssim \int |\psi(y)|(1+|y|)^{\frac nr} \frac{|f(x-y)|} { (1+|y|)^{\frac nr}} dy \notag \\
& \lesssim \sup_{y} \frac{|f(x-y)|} {(1+|y|)^{\frac nr}} \lesssim 
((\mathcal M (|f|^r))(x))^{\frac 1r},
\end{align*}
where in the last step we have used Lemma \ref{max_1t} (this is where the frequency localisation 
condition on $f$ is needed).
\end{proof}

\begin{rem*}
In Corollary \ref{cor_fmr}, the frequency localisation condition $\{|\xi|\lesssim 2^j\}$ is not needed for $1\le r<\infty$.
On the other hand, for $0<r<1$, this condition is indeed needed. For a simple counterexample, one can take (in 1D)
$f(y)= \chi_{\delta<|y|<1} |y|^{-1}$ for which $\mathcal M (|f|^r)$ is finite. On the other hand,
$(P_{\le 1} f)(0)$ can be made arbitrarily large by taking $\delta\to 0+$.

\end{rem*}

\begin{lem}\label{bd_h1}
Let $0<q<\infty$, $0<b<\infty$ and $c_1>0$. Let $\{f_j\}$, $j \in {\mathbb Z}$
be a sequence of functions such that
$\operatorname{supp}(\widehat{f_j}) \subset \{\xi:\, |\xi|<c_1 2^j\}$.
Then
\begin{align*}
\biggl\|\biggl(\sum_{j \in {\mathbb Z}}|\tilde P_j f_j|^b\biggr)^{\frac 1b}\biggr\|_{L^q}\lesssim_{q,b,c_1,n}
\biggl\|\biggl(\sum_{j \in {\mathbb Z}}|f_j|^b\biggr)^{\frac 1b}\biggr\|_{L^q}.
\end{align*}
\end{lem}

\begin{proof}
Note that by Corollary \ref{cor_fmr},
\begin{align*}
|(\tilde P_j f_j)(x)|
& \lesssim ( \mathcal M (|f_j|^r) ) (x))^{\frac 1r}.
\end{align*}

Then by taking $0<r<\min\{q,b\}$ and using Lemma \ref{lem_FS}, we get
\begin{align*}
\|  (\tilde P_j f_j )_{l_j^b}  \|_{L_x^q } & \lesssim \|  ((\mathcal M (
|f_j|^r) )^{\frac 1r} )_{l_j^b} \|_{L_x^q } \notag \\
 & \lesssim \| (\mathcal M (|f_j|^r) )_{l_j^{\frac br}} \|_{L_x^{\frac qr} }^{\frac 1r} \notag \\
 & \lesssim \|  (|f_j|^r )_{l_j^{\frac br}}  \|_{L_x^{\frac qr} }^{\frac 1r} 
  \lesssim \|  (f_j)_{l_j^b}  \|_{L_x^q }.
\end{align*}

\end{proof}

We will need to use the following useful estimates sometimes without explicit mentioning.

\begin{lem}\label{lm:eq} Let  $1\le q, r\le \infty$. If $s>0$, then
\begin{align}
\| (2^{js}  P_{>j} f)_{l_j^r} \|_{L_x^q} &\lesssim \| (2^{js} P_j f )_{l_j^r} \|_{L_x^q}; \label{1.1} \\
\| (2^{js} \|P_{>j} f \|_{L_x^q} ) \|_{l_j^r} & 
\lesssim \| (2^{js} \|P_j f \|_{L_x^q} ) \|_{l_j^r}. \label{1.1a} 
\end{align}
If $0<s<1$ and $1\le q<\infty$, then
\begin{align}
\| (2^{j(s-1)} \nabla P_{\le j} f )_{l_j^r} \|_{L_x^q}
\lesssim \; \| (2^{js} P_j f)_{l_j^r} \|_{L_x^q}. \label{2.1}
\end{align}
If $0<s<1$ and $1\le q\le \infty$, then
\begin{align}
\| (2^{j(s-1)}  \| \nabla P_{\le j} f  \|_{L_x^q} ) \|_{l_j^r}
\lesssim \; \| (2^{js} \|P_j f\|_{L_x^q} )\|_{l_j^r}. \label{1.4}
\end{align}
\end{lem}
\begin{proof} 
The inequality \eqref{1.1} follows from Young's inequality in $l^r$ space.  For \eqref{1.1a} one can first use triangle inequality
in $L_x^q$ and then use Young's inequality in $l^r$ space to conclude.

For \eqref{2.1}, just observe that by Young's inequality in $l^r$, we have
\begin{align*}
(2^{j(s-1)} \nabla P_{\le j} f)_{l_j^r} \lesssim (2^{js} \tilde P_j P_j f )_{l_j^r},
\end{align*}
where $\tilde P_j$ has frequency localised to $\{\xi:\, |\xi| \sim 2^j\}$.  The desired result
then follows from Lemma \ref{bd_h1} (here we need $q<\infty$). 

The proof of \eqref{1.4} is simpler and therefore omitted.

\end{proof}

\section{Regularity of pressure}
\begin{proof}[Proof of Theorem \ref{t0}]
\underline{Proof of \eqref{sob}}. We first notice that the case $s=0$ follows from \eqref{4}, the case
 $s=1$ follows from \eqref{4a} and standard bounds for Riesz type
operators. Therefore we may assume $0<s<1$ and estimate only the homogeneous $\dot W^{2s, q}$-norm of $p$.

Denote $R_{lk} = (-\Delta)^{-1} \partial_l \partial_k$ and make the decomposition
\begin{align}
p&=  \sum_{l,k}\sum_{j}  \biggl(R_{lk} ( P_{\le j-2} u_l P_j u_k) +
 R_{lk} ( P_j u_l  P_{\le j-2} u_k) +
 R_{lk} ( P_j u_l \tilde P_j u_k) \biggr) \notag \\
 &=: (A)+(B)+(C). \label{p_abc}
\end{align}
where $\tilde P_j=P_{j-1}+P_j+P_{j+1}$. 

For the first piece,  by using frequency localisation and the fact that $$ \sum_k \partial_k ( P_{\le j-2} u_l P_j u_k)
= \sum_{k} (\partial_k P_{\le j-2} u_l ) u_k,$$  we obtain
\begin{align*}
\| |\nabla|^{2s} (A) \|_{q}& \lesssim \| (2^{j(2s-1)}  | \nabla P_{\le j-2} u |  \cdot |P_j u | )_{l_j^2} \|_{q} \notag \\ 
&\lesssim \| (2^{j(s-1)} \nabla P_{\le j-2} u )_{l_j^2} \|_{2q} \cdot \| (2^{js} P_j u )_{l_j^2} \|_{2q} \lesssim \| u\|_{\dot W^{s,2q}}^2,
\end{align*}
where in the last step we have used Lemma \ref{lm:eq}.

The estimate of (B) is similar to (A) and therefore omitted.  For (C) we  write $u_l$ and $u_k$ simply as $u$, and estimate
\begin{align*}
\| |\nabla|^{2s}(C)\|_{q} & \lesssim \| ( 2^{2ms} | \sum_{j\ge m-10} P_j u \tilde P_j u |)_{l_m^2} \|_{q} \notag \\
&\lesssim \| (2^{js} P_j u)_{l_j^2} \|_{2q} \| (2^{js} \tilde P_j u )_{l_j^2} \|_{2q} \lesssim  \| u \|^2_{\dot W^{s,2q} }.
\end{align*}

\underline{Proof of \eqref{besov}}.  Note here $0<s<1$.  We use the same decomposition of $p$ as above. For the first
piece, one has
\begin{align*}
\| (A) \|_{\dot B^{2s}_{q,r} } & \lesssim \| (2^{j(2s-1)} \| \nabla P_{\le j-2} u \|_{2q} \| P_j u \|_{2q} ) \|_{l_j^r} \notag \\
& \lesssim \| (2^{j(s-1)} \| \nabla P_{\le j-2} u\|_{2q}) \|_{l_j^{2r} } \cdot \| u\|_{\dot B^{s}_{2q,2r}} 
\lesssim \| u \|^2_{\dot B^s_{2q,2r}}.
\end{align*}
The estimate of (B) is identical and omitted. For (C) we have (here we only need to require $s>0$)
\begin{align*}
\| (C) \|_{\dot B^{2s}_{q,r}} 
& \lesssim \| ( 2^{2ms} \| \sum_{j\ge m-10} P_j u \tilde P_j u \|_q )\|_{l_m^r} \lesssim
\| u \|_{\dot B^s_{2q,2r}}^2.
\end{align*}
\end{proof}
\begin{rem*}
The estimate \eqref{sob} can be easily derived using the new fractional Leibniz rule
introduced in \cite{L18}. Namely by using Corollary 1.4 in \cite{L18} and taking $s_1=s_2=s$,
$A^{2s}_{lk}= |\nabla|^{2s} (-\Delta)^{-1} \partial_l \partial_k$, and noting that $0<s<1$, we obtain
\begin{align*}
\| A^{2s}_{lk} (u_l u_k) -  u_kA^{2s}_{lk} u_l - u_l A^{2s}_{lk} u_k \|_{\frac q2}
\lesssim \| |\nabla|^s u_l \|_q \| |\nabla|^s u_k \|_q \lesssim \| |\nabla|^s u \|_q^2.
\end{align*}
Summing in $l$ and $k$ (inside the norm) and using incompressibility then yields the result.
\end{rem*}

\begin{proof}[Proof of Theorem \ref{t2}]
As was already mentioned before the case $0<s<1$ follows from \eqref{besov} and the embedding of $\dot B^{0}_{1,1}$ into 
$\mathcal H^1$.  Thus we only need to consider the case $s=1$.  By using \eqref{p_abc} and the same manipulations therein, we have
\begin{align*}
\|  |\nabla|^2 (A) \|_{\mathcal H^1}
&\lesssim \| (2^j |\nabla P_{\le j-2} u | \cdot |P_j u | )_{l_j^2} \|_1 
\lesssim \| (|\nabla P_{\le j-2} u|)_{l_j^{\infty}} \|_2 \| (2^j P_j u)_{l_j^2} \|_2 \notag \\
& \lesssim \| \nabla u \|_2^2.
\end{align*}
For the third piece, a close inspection of the Besov case in Theorem \ref{t0} shows that
\begin{align*}
\| (C) \|_{\dot B^2_{1,1} } \lesssim \| \nabla u \|_2^2.
\end{align*}
Since $\| |\nabla|^2 (C) \|_{\mathcal H^1} \lesssim \| |\nabla|^2(C) \|_{\dot B^0_{1,1}}$, the desired result follows easily. 
\end{proof}

\begin{proof}[Proof of Theorem \ref{dc}]
We first decompose the product $f\cdot g$ as
\begin{align*}
f\cdot g &= \underbrace{ \sum_{j\in \mathbb Z} f_{\le j-2} \cdot g_j}_{=:(a)} + 
\underbrace{\sum_{j\in \mathbb Z} f_j \cdot g_{\le j-2} }_{=:(b)} + 
\underbrace{\sum_{j\in \mathbb Z}
\tilde f_j \cdot  g_j}_{=:(c)}, \notag \\
\end{align*}
where $\tilde f_j = P_{j-1} f+ P_j f + P_{j+1} f$.  We only need to estimate $(a)$ and $(c)$ since the estimate of $(b)$ is
similar to $(a)$.  For $(a)$, note that by frequency localisation, we have
\begin{align*}
\| |\nabla|^s ( a) \|_{\mathcal H^1}  &\lesssim \| ( 2^{js} |f_{\le j-2} | \cdot |g_j |)_{l_j^2} \|_1\lesssim \| (2^{js} g_{j} )_{l_j^2} \|_{p_2}
\cdot \| (f_{\le j-2})_{l_j^{\infty}} \|_{p_2^{\prime}}  \notag \\
& \lesssim\; \| |\nabla|^s g \|_{p_2} \cdot \| f\|_{p_2^{\prime}}.
\end{align*}
One should note that for this estimate we do not have severe constraints on the exponent $s$. 
On the other hand if $-1<s\le 0$, then
\begin{align*}
\| |\nabla|^s (a) \|_{\mathcal H^1} 
& \lesssim \| (2^{js} |f_{\le j-2} | )_{l_j^\infty} \|_{p_1} \| (g_j)_{l_j^{2}} \|_{p_1^{\prime}}
\lesssim \| |\nabla|^s f \|_{p_1} \| g \|_{p_1^{\prime}}.
\end{align*}
Thus
\begin{align*}
\| |\nabla|^s(a) \|_{\mathcal H^1} \lesssim
\begin{cases} 
 \| |\nabla|^s g \|_{p_2} \cdot \| f\|_{p_2^{\prime}}, \quad \text{if $s>0$;} \\
\min\{ \| |\nabla|^s f \|_{p_1} \| g \|_{p_1^{\prime}}, \; \| |\nabla|^s g \|_{p_2} \| f \|_{p_2^{\prime}} \},
\quad \text{if $-1<s\le 0$.}
\end{cases}
\end{align*}

Now for the diagonal piece $(c)$, we first note that for $s>0$ the ``Div-Curl" condition is not needed since
\begin{align*}
\| |\nabla|^s(c) \|_{\mathcal H^1} & \lesssim \| |\nabla|^s(c) \|_{\dot B^0_{1,1}} \notag \\
& \lesssim \sum_{k} 2^{ks} \sum_{j\ge k-10} \| \tilde f_j g_j \|_1 \notag \\
&\lesssim
\min\{ \| |\nabla|^s f \|_{p_1} \| g \|_{p_1^{\prime}}, \; \| |\nabla|^s g \|_{p_2} \| f \|_{p_2^{\prime}} \}.
\end{align*}

In the regime $-1<s\le 0$ we need to exploit the ``Div-Curl" condition. 
Since $\nabla \times g=0$ we can write
$g= \nabla \psi$ for a scalar potential $\psi$. Since $\nabla \cdot f=0$, we have
\begin{align*}
(c)= \sum_j \tilde f_j \cdot g_j = \sum_j \nabla\cdot ( \tilde f_j  \psi_j).
\end{align*}
Clearly by frequency localisation,
\begin{align*}
\| |\nabla|^s (c) \|_{\mathcal H^1} &\lesssim \| |\nabla|^s(c) \|_{\dot B^0_{1,1}}
\lesssim  \| (2^{k(s+1)} \sum_{j\ge k-10} |f_j| \cdot |\psi_j| )_{l_k^1} \|_1 \notag \\
&\lesssim \| (2^{j(s+1)} |f_j| \cdot |\psi_j|)_{l_j^1} \|_1 \notag \\
&\lesssim \min\{ \| |\nabla|^s f \|_{p_1} \| g \|_{p_1^{\prime}}, \; \| |\nabla|^s g \|_{p_2} \| f \|_{p_2^{\prime}} \}.
\end{align*}
\end{proof}

\begin{rem*}
We record here a quite useful (albeit simple) estimate extracted from the proof above. For any 
$\varphi\in C_c^{\infty}(\mathbb R^n)$, $\psi \in C_c^{\infty} (\mathbb R^n)$ with support 
\emph{localised to an
annulus} (say $\{|\xi| \sim 1\}$), define 
\begin{align*}
\widehat{P_j^{\varphi} f} (\xi) = \varphi(2^{-j} \xi) \hat f(\xi), 
\quad \widehat{P_j^{\psi} f} (\xi) = \psi(2^{-j} \xi) \hat f(\xi).
\end{align*}
Then  for any $s>0$, any $s_1,s_2\ge 0$ with $s_1+s_2=s$,
and any $f \in \mathcal S(\mathbb R^n)$, $g \in \mathcal S(\mathbb R^n)$, we have
\begin{align}
\| |\nabla|^s ( \sum_j P_j^{\varphi} f
P_j^{\psi} g ) \|_{\mathcal H^1} & \lesssim 
\| \sum_j P_j^{\varphi} f P_j^{\psi} g \|_{\dot B^s_{1,1}}  \notag \\
&\lesssim_{s,\varphi,\psi,n,r,s_1,s_2}
\,  \| |\nabla|^{s_1} f \|_{r} \| |\nabla|^{s_2} g\|_{r^{\prime}},   \label{pj_diagonal}
\end{align}
where $1<r<\infty$, $r^{\prime}=\frac r {r-1}$. 

Similarly for any $s>0$, any $s_1,s_2\ge0$ with $s_1+s_2=s$, any
$1<p<\infty$, any $1<p_1,p_2< \infty$ with $\frac 1p =\frac 1{p_1}+\frac 1{p_2}$, 
we have
\begin{align}
\| |\nabla|^s ( \sum_j P_j^{\varphi} f
P_j^{\psi} g ) \|_{p} &\lesssim_{s,\varphi,\psi,n,p,p_1,p_2,s_1,s_2}
\,  \| |\nabla|^{s_1} f \|_{p_1} \| |\nabla|^{s_2} g\|_{p_2}.   \label{pj_diagonal1}
\end{align}
Also (the following corresponds to $p_2=\infty$)
\begin{align}
\| |\nabla|^s ( \sum_j P_j^{\varphi} f
P_j^{\psi} g ) \|_{p} &\lesssim_{s,\varphi,\psi,n,p,s_1,s_2}
\,  \| |\nabla|^{s_1} f \|_{p} \| |\nabla|^{s_2} g\|_{\dot B^0_{\infty,\infty} }.   \label{pj_diagonal2}
\end{align}
\end{rem*}

Inspired by the ``Div-Curl" theorem and the considerations above, we record below a simple fractional Leibniz rule in the frequency-localised context which is
quite useful for practical purposes.
For $s>0$,  denote by $A^s$  a differential operator such that its symbol $\widehat{A^s}(\xi)$  is a homogeneous
function of degree $s$ and $\widehat{A^s}(\xi) \in C^{\infty}(\mathbb S^{n-1})$ (for example:
$\widehat{A^s}(\xi) = i |\xi|^{s-1} \xi_1$ which corresponds to $A^s= |\nabla|^{s-1} \partial_1$, or
$\widehat{A^s}(\xi)=|\xi|^s$ which corresponds to $A^s=|\nabla|^s=D^s$).  We have the following proposition.

\begin{prop} \label{prop_H1case}
Let $s>0$. Then for any $1<p_1,p_2<\infty$ with $\frac 1{p_1}+\frac 1 {p_2}=1$,
any $s_1,s_2\ge 0$ with $s_1+s_2=s$, and any $f$, $g\in \mathcal S(\mathbb R^n)$, we have
\begin{align*}
\|A^s (fg)  &- \sum_{j \in \mathbb Z} A^s (f_{\le j-2} g_j) - \sum_{j \in \mathbb Z}
A^s (f_j g_{\le j-2} ) \|_{\mathcal H^1} \notag \\
& \lesssim \|A^s (fg)  - \sum_{j \in \mathbb Z} A^s (f_{\le j-2} g_j) - \sum_{j \in \mathbb Z}
A^s (f_j g_{\le j-2} ) \|_{\dot B^0_{1,1} } \notag \\
&\lesssim_{A^s, s,s_1,s_2, p,p_1,p_2, n}
\| D^{s_1} f \|_{p_1} \|D^{s_2} g \|_{p_2}.
\end{align*}
For any $1<p<\infty$,  any $1<p_1,p_2<\infty$ with $\frac 1{p_1}+\frac 1 {p_2}=\frac 1p$,
any $s_1,s_2\ge 0$ with $s_1+s_2=s$, and any $f$, $g\in \mathcal S(\mathbb R^n)$, we have
\begin{align*}
\|A^s (fg)  &- \sum_{j \in \mathbb Z} A^s (f_{\le j-2} g_j) - \sum_{j \in \mathbb Z}
A^s (f_j g_{\le j-2} ) \|_{p} \notag \\
&\lesssim_{A^s, s,s_1,s_2, p,p_1,p_2, n}
\| D^{s_1} f \|_{p_1} \|D^{s_2} g \|_{p_2}.
\end{align*}
Also
\begin{align*}
\|A^s (fg)  &- \sum_{j \in \mathbb Z} A^s (f_{\le j-2} g_j) - \sum_{j \in \mathbb Z}
A^s (f_j g_{\le j-2} ) \|_{p} \notag \\
&\lesssim_{A^s, s,s_1,s_2, p, n}
\| D^{s_1} f \|_{p} \|D^{s_2} g \|_{\dot B^0_{\infty,\infty} }.
\end{align*}

\end{prop}
\begin{rem*}
Since $f$, $g\in \mathcal S(\mathbb R^n)$, one does not need to worry about the convergence issues. For example, it
is easy to check that $\sum_j f_{\le j-2} g_j \in H^m(\mathbb R^n)$, $\sum_j \tilde f_j g_j \in W^{m,2}(\mathbb R^n)$, for all $m\ge 0$,
and 
\begin{align*}
&A^s (\sum_{j} f_{\le j-2} g_j) = \sum_j A^s (f_{\le j-2} g_j), \notag \\
& A^s( \sum_j \tilde f_j g_j ) = \sum_j A^s (\tilde f_j g_j),
\end{align*}
where the above equality holds point-wisely for $x\in \mathbb R^n$ and in stronger norms.
\end{rem*}
\begin{proof}[Proof of Proposition \ref{prop_H1case}]
This follows from the para-product decomposition, \eqref{pj_diagonal}, \eqref{pj_diagonal1}
and \eqref{pj_diagonal2}.
\end{proof}

\begin{cor}
Let $s>0$. For any $1<r_1,r_2<\infty$, we have
\begin{align*}
\| A^s (fg )\|_{\mathcal H^1} \lesssim_{A^s, s,n,r_1,r_2}\; \| D^s f \|_{r_1} \| g \|_{\frac {r_1} {r_1-1}} + \| f \|_{r_2}
\|D^s g \|_{\frac {r_2} {r_2-1}}.
\end{align*}
For any $1<p<\infty$, any $1<p_1,p_3 <\infty$, $1< p_2,p_4 \le \infty$ with $
\frac 1 p=\frac 1 {p_1}+\frac 1 {p_2} = \frac 1 {p_3}+\frac 1 {p_4}$, we have
\begin{align*}
\| A^s (fg) \|_p \lesssim_{A^s,s,n,p,p_1,p_2,p_3,p_4}\; \| D^s f \|_{p_1} \| g\|_{p_2} + \| D^s g\|_{p_3} \| f \|_{p_4}.
\end{align*}

\end{cor}
\begin{proof}
This follows easily from Proposition \ref{prop_H1case}. For example, one observes that 
\begin{align*}
\| \sum_j A^s (f_{\le j-2} g_j) \|_p \lesssim \| (2^{js} g_j)_{l_j^2} \|_{p_3} \| (f_{\le j-2})_{l_j^{\infty}}  \|_{p_4}
\lesssim \| D^s g \|_{p_3} \| f \|_{p_4}.
\end{align*}
The other cases are similarly estimated. We omit the details.
\end{proof}

\section{counterexamples}


\begin{proof}[Proof of Theorem \ref{prop1}, case $s=1$.]

Step 1: Some simplification. We shall take
$u=(u_1,u_2,0,0,\cdots,0)$, where
\begin{align*}
u_1 = -\partial_2 \phi, \quad u_2 = \partial_1 \phi,
\end{align*}
and $\phi:\, \mathbb R^n \to \mathbb R$ will be chosen later. Note
that by construction $u$ is divergence-free. Now
\begin{align}
-\Delta p &= \sum_{j,k} \partial_j u_k \partial_k u_j \notag \\
& = (\partial_1 u_1)^2 +(\partial_2 u_2)^2 + 2\partial_1 u_2
\partial_2 u_1 \notag \\
& = 2 \Bigl( (\partial_{12} \phi)^2 - \partial_{11} \phi
\partial_{22} \phi \Bigr). \label{pphi_100}
\end{align}

Clearly, the task now is to find a scalar function $\phi \in
\mathcal S(\mathbb R^n)$ with $\|\phi\|_2 + \| \phi\|_{\dot B^2_{\infty,\infty}} \le
1$, such that
\begin{align*}
\| (\partial_{12} \phi)^2 - \partial_{11} \phi
\partial_{22} \phi \|_{\dot B^0_{\infty,\infty} } \gg 1.
\end{align*}
As a matter of fact, we shall show that for some $j_1 \gg 1$, 
\begin{align*}
\| P_{j_1} \bigl( (\partial_{12} \phi)^2 - \partial_{11} \phi
\partial_{22} \phi  \bigr)\|_{\infty} \gg 1,
\end{align*}
i.e. the norm inflation occurs at a high frequency block.

Step 2: There exists $\phi_0 \in \mathcal S(\mathbb R^n)$ with
$\operatorname{supp}(\widehat{\phi_0}) \subset\{ \, \xi: \; \frac 23
\le |\xi| \le \frac 56\}$ and
\begin{align*}
\Bigl( (\partial_{12} \phi_0)^2 - \partial_{11} \phi_0
\partial_{22} \phi_0 \Bigr)(0) \ne 0.
\end{align*}

Indeed one can take $\widehat{\phi_0}(\xi)$ as a radial real-valued
function in $\xi$ with
\begin{align*}
\int \xi_1 \xi_2 \widehat{\phi_0}(\xi) d\xi =0, \\
\int \xi_1^2 \widehat{\phi_0} (\xi) d\xi \ne 0, \\
\int \xi_2^2 \widehat{\phi_0}(\xi) d\xi \ne 0.
\end{align*}
These yield $(\partial_{12} \phi_0)(0)=0$, $(\partial_{11}
\phi_0)(0) \cdot (\partial_{22} \phi_0)(0) \ne 0$.

Step 3: Now take $J\gg 1$, $J$ being an even integer and
\begin{align*}
\phi(x) = \sum_{j=100}^J 2^{-2 j^2} \phi_0(2^{j^2} x).
\end{align*}
Easy to check that (note that $\operatorname{supp}(\hat \phi_0) \subset
\{ \frac 2 3 \le |\xi| \le \frac 5 6\} \subset
\{ \frac 12 <|\xi|< \frac 7 6 \}$)
\begin{align}
\| \phi\|_2 +\| \phi\|_{\infty}+ \| \phi \|_{\bs 2} \lesssim 1.
\end{align}

Then clearly
\begin{align}
&(\partial_{12} \phi (x))^2 \notag \\
= & \sum_{j=100}^J (\partial_{12} \phi_0)(2^{j^2} x)  (\partial_{12}
\phi_0)(2^{j^2} x) \notag \\
& + 2 \sum_{j_1=100}^J (\partial_{12} \phi_0)(2^{j_1^2} x) \bigl(
\sum_{100\le j_2<j_1} (\partial_{12} \phi_0)(2^{j_2^2} x) \bigr).
\end{align}

Similarly
\begin{align}
&  \partial_{11} \phi(x) \partial_{22} \phi(x) \notag \\
= & \sum_{j=100}^J (\partial_{11} \phi_0)(2^{j^2} x) (\partial_{22}
\phi_0 )(2^{j^2} x) \notag \\
& + \sum_{j_1=100}^J(\partial_{11} \phi_0)(2^{j_1^2} x) \bigl(
\sum_{100\le j_2<j_1} (\partial_{22} \phi_0)(2^{j_2^2} x) \bigr)
\notag \\
&   + \sum_{j_1=100}^J(\partial_{22} \phi_0)(2^{j_1^2} x) \bigl(
\sum_{100\le j_2<j_1} (\partial_{11} \phi_0)(2^{j_2^2} x) \bigr).
\end{align}

Thus
\begin{align}
& (\partial_{12} \phi(x) )^2 - \partial_{11} \phi(x) \partial_{22}
\phi(x) \notag \\
= & H_1(x)+H_2(x),
\end{align}
where
\begin{align}
H_1(x) = \sum_{j=100}^J \biggl( (\partial_{12}
\phi_0)^2-\partial_{11}\phi_0 \partial_{22} \phi_0 \biggr)(2^{j^2}
x),
\end{align}
is the ``diagonal piece", and $H_2(x)$ is the off-diagonal piece.

We shall show
\begin{align*}
\| P_{\le 1} H_1\|_{\infty}+ \| H_1\|_{\bs 0} \lesssim 1, \quad \| P_{j_1} H_2 \|_{\infty} \gg 1, \quad
\text{for some $j_1\gg 1$}.
\end{align*}

Step 4: the diagonal piece. One can rewrite $H_1$ as
\begin{align*}
2H_1(x)& = \partial_1 \Bigl( \sum_{j=100}^J  2^{-j^2} ( \partial_2
\phi_0
\partial_{12} \phi_0 -\partial_1 \phi_0 \partial_{22} \phi_0
)(2^{j^2} x) \Bigr) \notag \\
&\; + \partial_2 \Bigl( \sum_{j=100}^J 2^{-j^2} ( -\partial_2 \phi_0
\partial_{11} \phi_0 + \partial_1 \phi_0 \partial_{12} \phi_0
)(2^{j^2} x) \Bigr).
\end{align*}

Thanks to the above expression, it is easy to check that
\begin{align*}
\| P_{\le 1} H_1 \|_{\infty} \lesssim \sum_{j\ge 100} 2^{-j^2} \lesssim 1.
\end{align*}

Note that on the Fourier side, we have roughly speaking 
\begin{align*}
\widehat{H_1}(\xi) \sim \xi \sum_{j=100}^J c_j 1_{|\xi| \le \frac 53 2^{j^2}}.
\end{align*}

It is then easy to check that for each $l\ge 2$ (note that $P_l$ has frequency localised to $\frac 12 2^{l} \le |\xi| \le \frac 76 2^{l}$),
\begin{align*}
\| P_l H_1\|_{\infty} \lesssim  2^l \sum_{j\ge \sqrt {l-2}}
2^{-j^2} \lesssim 1.
\end{align*}

Step 5: The off-diagonal piece. Take $j_0=J/2$ and $j_1={j_0^2}$. It
is not difficult to check that
\begin{align*}
P_{j_1} H_2& = 2 (\partial_{12} \phi_0)(2^{j_0^2}x) \sum_{100\le j_2
<j_0} (\partial_{12} \phi_0)(2^{j_2^2} x) \notag \\
&\;\;-(\partial_{11} \phi_0)(2^{j_0^2} x)  \sum_{100\le j_2<j_0}
(\partial_{22} \phi_0)(2^{j_2^2} x)
\notag \\
&\;\;   - (\partial_{22} \phi_0)(2^{j_0^2} x)  \sum_{100\le j_2<j_0}
(\partial_{11} \phi_0)(2^{j_2^2} x).
\end{align*}

Thus
\begin{align*}
|(P_{j_1} H_2)(0)| \gtrsim j_0 | (\partial_{12} \phi_0)(0)^2 -
(\partial_{11} \phi_0)(0) (\partial_{22} \phi_0)(0)| \gg 1.
\end{align*}

This yields the desired result.
\end{proof}

\begin{proof}[Proof of Theorem \ref{prop1}, case $s=0$.]
We shall take $u= (-\partial_2 \phi, \partial_1 \phi, 0,\cdots, 0)$ where $\phi:\; \mathbb R^n \to \mathbb R$ will
be specified momentarily. 

Step 1. First we choose Schwartz $g:\, \mathbb R^n \to \mathbb R$ with
$\operatorname{supp}(\hat g) \subset\{ \xi:\, \frac 12 \le |\xi| \le 2 \}$ such that 
\begin{align*}
(P_{j_0} \Delta^{-1} \partial_{11} (g^2) \Bigr)(0 ) \ne 0, \qquad \text{for any $j_0 \in \mathbb Z$}.
\end{align*}
This can be easily achieved since $\Delta^{-1} \partial_{11}$ corresponds to the symbol $\xi_1^2/ |\xi|^2$ and one can just
take $\hat g $ to be non-negative.

Step 2.  Let $J\gg 1$ and define
\begin{align*}
\phi(x)= \sum_{j=1}^J \frac 1 {k_j} \cos (k_j x_2) g(x),
\end{align*}
where $k_j$ are dyadic numbers which are well separated and will be chosen sufficiently large. 
Note that $\| \phi \|_{\dot B^1_{\infty,\infty}}  + \| \phi\|_2  \lesssim 1$.

Step 3. Non-diagonal term does not matter.  Note that $\|u \|_{\dot B^0_{\infty,\infty}} \lesssim 1$. Recall
\begin{align*}
-p = \sum_{l,k} \Delta^{-1} \partial_l \partial_k( u_l u_k).
\end{align*}
Clearly
\begin{align*}
 &\| \sum_{l,k} \Delta^{-1} \partial_l \partial_k(  \sum_{j } P_{\le j-2} u_l P_j u_k ) \|_{\dot B^0_{\infty,\infty}} \notag \\
\lesssim\; & \sup_{j}   \sum_{l,k}  \|\Delta^{-1} \partial_l   ( \partial_k P_{\le j-2} u_l  P_j u_k) \|_{\infty} \lesssim 
\| u\|_{\dot B^0_{\infty,\infty}}^2\lesssim 1.
\end{align*}

Step 4. Inflation through the diagonal terms. In terms of $\phi$, we have
\begin{align*}
-p = \Delta^{-1} \partial_{11} ( (\partial_2 \phi)^2) + 
\Delta^{-1} \partial_{22} (  (\partial_1 \phi)^2) - 2 \Delta^{-1} \partial_{12} ( \partial_1 \phi \partial_2 \phi).
\end{align*}
Clearly
\begin{align*}
\partial_2 \phi = \sum_{j=1}^J
\bigl( -\sin(k_j x_2) g(x) + \frac 1 {k_j} \cos (k_j x_2)  \partial_2 g \bigr).
\end{align*}
Then
\begin{align*}
(\partial_2 \phi)^2 & = \text{``Non-diagonal terms"} \notag \\
& \quad + \sum_{j=1}^J \bigl( \sin^2(k_j x_2) g(x)^2 + \frac 1 {k_j^2} \cos^2(k_j x_2) (\partial_2 g)^2 \bigr) \notag \\
& =\text{``Non-diagonal terms"} \notag \\
& \quad + \frac 12 J g(x)^2 + \frac 12 ( \sum_{j=1}^J \frac 1 {k_j^2} ) (\partial_2 g)^2  \notag \\
& \quad + \frac 12 \sum_{j=1}^J
\bigl( -\cos(2k_j x_2) g(x)^2+ \frac 1 {k_j^2} \cos (2k_j x_2) (\partial_2 g)^2  \bigr).
\end{align*}
It follows that
\begin{align*}
\| -p -  \frac J2 \Delta^{-1} \partial_{11} (g^2) \|_{\dot B^0_{\infty,\infty}} \lesssim 1.
\end{align*}
Thus for any $j_0\in \mathbb Z$, by taking $J$ large, we obtain
\begin{align*} 
|(P_{j_0} p)(0)| \gtrsim \operatorname{const}\cdot J \gg 1.
\end{align*}
This clearly implies the desired inflation of $\| p\|_{\dot B^0_{\infty,\infty}}$ in high frequency.
\end{proof}

We now consider the case $u \in \dot C^{\frac 12}=\bs{\frac 12}$.

\begin{prop} \label{prop3}
For any $\epsilon>0$, there exists divergence free $u \in \mathcal
S(\mathbb R^n)$ with $\| u \|_{\bs{\frac 12}} +\|u\|_2\le 1$ such
that
\begin{align*}
\| \nabla p \|_{\infty} >\frac 1 {\epsilon}.
\end{align*}
\end{prop}
\begin{proof}[Proof of Proposition \ref{prop3}] $$\;$$
Step 1:
 Take a scalar function $\phi_0 \in \mathcal S(\mathbb R^n)$
with $\operatorname{supp}(\widehat{\phi_0}) \subset\{ \, \xi: \;
\frac 23 \le |\xi| \le \frac 56\}$. Let $J\gg 1$ and
\begin{align*}
\phi(x) = \sum_{j=1000}^J N_j^{-\frac 32} \phi_0( N_j x),
\end{align*}
where $N_j= 2^{a_j}$, $a_j=3^{j^2}$. Note that by construction we
have
\begin{align} \label{prop3_e4}
N_{j-1} < N_j^{\frac 12}, \qquad\forall\, j.
\end{align}
Denote $\nabla^{\perp}=(-\partial_2, \partial_1,0,\cdots,0)$. Then
define
\begin{align*}
u= \nabla^{\perp} \phi= \sum_{j=1000}^J  \underbrace{ N_j^{-\frac
12} (\nabla^{\perp} \phi_0)(N_j x) }_{=:v_j}.
\end{align*}

  Easy
to check that
\begin{align*}
\| u \|_{\bs {\frac 12}} + \|u\|_2 \lesssim 1.
\end{align*}

On the other hand,
\begin{align}
\nabla p &= (-\Delta)^{-1} \nabla \nabla \cdot ( (u\cdot \nabla) u)
\notag \\
&= \sum_{j} (-\Delta)^{-1} \nabla \nabla \cdot ( (v_j \cdot \nabla)
v_j ) \label{prop3_e5a} \\
& \quad + \sum_j  (-\Delta)^{-1} \nabla \nabla \cdot (  (v_j \cdot
\nabla) v_{\le j-1} ) \label{prop3_e5b} \\
& \quad + \sum_j (-\Delta)^{-1} \nabla \nabla \cdot ( (v_{\le j-1}
\cdot \nabla )v_j), \label{prop3_e5c}
\end{align}
where we have denoted $v_{\le j-1} = \sum_{i\le j-1} v_i$.

By \eqref{prop3_e4}, frequency localization and Bernstein, we have
\begin{align*}
\| \eqref{prop3_e5b} \|_{\infty} \lesssim \sum_j N_j^{-\frac 12}
N_{j-1}^{\frac 12} \lesssim 1.
\end{align*}

For \eqref{prop3_e5c}, we shall transform it to a similar form as in
\eqref{prop3_e5b}. Note that for any two vector functions $f$, $g$
with $\nabla \cdot f=\nabla \cdot g=0$, we have
\begin{align*}
\nabla \cdot ( (f\cdot \nabla) g) & = \sum_{k,l} \partial_k f_l
\partial_l g_k \notag \\
& = \sum_{l} \partial_l ( \sum_k (\partial_k f_l) g_k).
\end{align*}
Thanks to the above transformation, it is clear that we have
\begin{align*}
\nabla \cdot ( (v_{\le j-1} \cdot \nabla )v_j ) = O ( \partial (
\partial v_{\le j-1} v_j)).
\end{align*}
Thus it can be estimated in the same way as in \eqref{prop3_e5b},
and we have
\begin{align*}
\| \eqref{prop3_e5c} \|_{\infty} \lesssim_j N_j^{-\frac 12} N_{j-1}^{\frac 12}
\lesssim 1.
\end{align*}

Step 2: From step 1, it is clear that the dominant contribution to
$\|\nabla p\|_{\infty}$ comes from the diagonal piece
\eqref{prop3_e5a}. Denote \eqref{prop3_e5a} as $H(x)$. Then by using
scaling, we have
\begin{align*}
H(0)  & = \sum_j  \Bigl( (-\Delta)^{-1} \nabla \nabla \cdot
((v_j\cdot \nabla) v_j ) \Bigr)(0) \notag \\
& =(J-999)  \Bigl( (-\Delta)^{-1} \nabla \nabla\cdot(
(\nabla^{\perp} \phi_0) \cdot \nabla)(\nabla^{\perp} \phi_0)
\Bigr)(0).
\end{align*}
Thus we only need to show the existence of $\phi_0$ such that
\begin{align*}
|\Bigl( (-\Delta)^{-1} \nabla \nabla\cdot( (\nabla^{\perp} \phi_0)
\cdot \nabla)(\nabla^{\perp} \phi_0) \Bigr)(0)| \ne 0.
\end{align*}
This is not difficult to do and we sketch the detail below.

First by using the proof of Theorem \ref{prop1} (see Step 2
therein), one can find $\tilde \phi_0 \in \mathcal S(\mathbb R^n)$
with
$\operatorname{supp}(\widehat{\tilde \phi_0}) \subset\{ \, \xi: \; \frac 23
\le |\xi| \le \frac 56\}$ 
such that
\begin{align*}
 & \Bigl(\nabla \cdot ( (\nabla^{\perp}\tilde \phi_0 \cdot
 \nabla)(\nabla^{\perp} \tilde \phi_0) ) \Bigr)(0)  \notag \\
 = &\; 2 \Bigl( ( \partial_{12} \tilde \phi_0)^2 -\partial_{11}
 \tilde \phi_0 \partial_{22} \tilde \phi_0  \Bigr)(0) \ne 0.
 \end{align*}
 Take a smooth bump function $g$ such that
 \begin{align*}
\int_{\mathbb R^n} \Bigl(\nabla \cdot ( (\nabla^{\perp}\tilde \phi_0
\cdot
 \nabla)(\nabla^{\perp} \tilde \phi_0) ) \Bigr)(x) g(x) dx \ne 0.
 \end{align*}
 Clearly then
 \begin{align*}
\int_{\mathbb R^n} \Bigl( (-\Delta)^{-1} \nabla \nabla \cdot (
(\nabla^{\perp}\tilde \phi_0 \cdot
 \nabla)(\nabla^{\perp} \tilde \phi_0) ) \Bigr)(x) \cdot \nabla  g(x) dx \ne 0.
 \end{align*}
From this, one sees that there exists $x_*$ such that
\begin{align*}
\Bigl( (-\Delta)^{-1} \nabla \nabla \cdot ( (\nabla^{\perp}\tilde
\phi_0 \cdot
 \nabla)(\nabla^{\perp} \tilde \phi_0) ) \Bigr)(x_*) \ne 0.
 \end{align*}
Define $\phi_0(x) = \tilde \phi_0(x+x_*)$. Since translation in
physical space is equivalent to modulation in frequency space (which
does not change frequency localisation), it is not difficult to
check that $\phi_0$ satisfies all the required properties.
\end{proof}

The next proposition requires a slightly different construction, in
the sense that the frequency supports are no longer disjoint. We
only need to work in the physical space.

\begin{prop} \label{prop4}
For any $\epsilon>0$, there exists divergence free $u \in \mathcal
S(\mathbb R^n)$ with $\| \nabla u \|_{\infty} +\|u\|_2\le 1$ such
that
\begin{align*}
\| \nabla p \|_{C^1} >\frac 1 {\epsilon}.
\end{align*}
\end{prop}
\begin{proof}[Proof of Proposition \ref{prop4}]
Take a smooth bump function $\phi_0 \in C_c^{\infty}(\mathbb R^n)$
with $$\operatorname{supp}(\phi_0) \subset\{ x:\; c_1
<|x|<c_2 \}$$ for some constants $0<c_1<c_2<\infty$. Define
\begin{align*}
\phi(x) = \sum_{j=1}^J k_j^{-2} \phi_0(k_j x),
\end{align*}
where  $J\gg 1$, and $k_j$ are dyadic numbers sufficiently large such that each summand has disjoint support. Denote $\nabla^{\perp}= (-\partial_2, \partial_1,
0,\cdots, 0)$ and
\begin{align*}
u = \nabla^{\perp} \phi = \sum_{j=1}^J \underbrace{k_j^{-1}
(\nabla^{\perp} \phi_0)(k_j x)}_{=:v_j}.
\end{align*}
Obviously $\|u\|_{C^1} + \|u\|_2 \lesssim 1$ (since the rescaled
copies of $\nabla^{\perp} \phi_0$ have disjoint supports). On the
other hand, by using scaling,
\begin{align*}
(\partial_{11} p)(0) & = \Bigl( (-\Delta)^{-1} \partial_{11} \nabla \cdot (
(u\cdot \nabla ) u ) \Bigr)(0) \notag \\
& = \sum_{j=1}^J \Bigl( (-\Delta)^{-1} \partial_{11} \nabla \cdot ( v_j
\cdot \nabla v_j ) \Bigr)(0) \notag \\
& = J \Bigl( (-\Delta)^{-1} \partial_{11} \nabla \cdot (
(\nabla^{\perp} \phi_0 \cdot \nabla) \nabla^{\perp} \phi_0)
\Bigr)(0).
\end{align*}

Thus we only need to choose $\phi_0$ compactly supported on an annulus such that
\begin{align*}
|\Bigl( (-\Delta)^{-1} \partial_{11} \nabla \cdot ( (\nabla^{\perp} \phi_0
\cdot \nabla) \nabla^{\perp} \phi_0) \Bigr)(0)|\ne 0.
\end{align*}
This is not difficult to achieve and we sketch the details below. 

Step 1. Simplification. We shall choose $\phi_0$ to have support away from the origin. This way one does not need to worry
about the singular kernels in the integration by part argument below.
Denote $v=\nabla^{\perp} \phi_0$. Then
\begin{align*}
  & \bigl( \Delta^{-1} \partial_{11} \nabla \cdot ( v\cdot \nabla v) \bigr)(0) \notag \\
  =&\; \sum_{1\le j,k\le 2} \int v_j v_k \partial_j \partial_k (\; \underbrace{\partial_{11} \Delta^{-1} \delta_0}_{=:a(x)}\;) dx \notag \\
  =&\; \int ( v_1^2  \partial_{11} a + v_2^2 \partial_{22} a -2 v_1 v_2 \partial_{12} a ) dx.
  \end{align*}

Step 2. Case dimension $n=2$. Note that 
\begin{align*}
a(x)= C_1 \cdot ( (x_1^2+x_2^2)^{-1} - 2 
(x_1^2+x_2^2)^{-2} x_1^2),
\end{align*}
where $C_1>0$ is an absolute constant.  Denote $\tilde a(x) = a(x)/C_1$. Easy to verify that
\begin{align*}
(\partial_{11} \tilde a)(1,0)=-6, \quad (\partial_{22} \tilde a)(1,0)=6, \quad (\partial_{12} \tilde a)(1,0)=0.
\end{align*}
Now define $x_{\ast}=(1,0)$ and let
\begin{align*}
\phi_0(x) = \cos (kx_2) c(\frac {x-x_{\ast}} {\eta}),
\end{align*}
where $c$ is a bump function localised to the unit ball, $\eta>0$ will be taken sufficiently small (so that $|(\partial_{11} 
\tilde a)(x)| \sim 1$ for $|x-x_{\ast}|<\eta$)
and $k$ will be taken sufficiently large. It is then easy to check (recall $v_1=-\partial_2 \phi_0$)
\begin{align*}
&\left|\int ( v_1^2  \partial_{11} \tilde a + v_2^2 \partial_{22} \tilde a -2 v_1 v_2 \partial_{12} \tilde a ) dx\right| \notag \\
\ge &\; \operatorname{const} \cdot k^2 - O(k) \gtrsim 1,
  \end{align*}
  if $k$ is sufficiently large. This settles the 2D case.

Step 3. Case dimension $n\ge 3$.  Denote $x_{\ast}=(1,0,\cdots,0)$. Easy to check that
\begin{align*}
\partial_{11} a(x) \bigr|_{x=x_{\ast}} = - \operatorname{const} \cdot \partial_{x_1}^4 ( x_1^{-(n-2) } ) \Bigr|_{x_1=1} \ne 0.
\end{align*}
Again  let
\begin{align*}
\phi_0(x) = \cos (kx_2) c(\frac {x-x_{\ast}} {\eta})
\end{align*}
will yield
\begin{align*}
&\left|\int ( v_1^2  \partial_{11} \tilde a + v_2^2 \partial_{22} \tilde a -2 v_1 v_2 \partial_{12} \tilde a ) dx\right| \notag \\
\ge &\; \operatorname{const} \cdot k^2 - O(k) \gtrsim 1.
  \end{align*}
  This settles the case for $n\ge 3$. 
\end{proof}

\begin{proof}[Proof of Proposition \ref{prop5}]
Define $\nabla^{\perp}=(-\partial_2,\partial_1, 0,\cdots,0)$. Let
$\phi \in C_c^{\infty}(\mathbb R^n)$ be such that
\begin{align*}
&\int_{\mathbb R^n} (\partial_1 \phi)^2 dx \ne \int_{\mathbb R^n}
(\partial_2 \phi)^2 dx, \\
&\int_{\mathbb R^n} \partial_1 \phi \partial_2 \phi dx =0.
\end{align*}
For example, one can take $\phi(x) =b(x) \cos (kx_1)$, where $b \in C_c^{\infty}(\mathbb R^n)$ is radial, and $k$ is sufficiently large.

Define $u= \nabla^{\perp} \phi$.  By using a computation similar to that in \eqref{pphi_100},
we have
\begin{align*}
p = 2(-\Delta)^{-1} \Bigl( \underbrace{(\partial_{12} \phi)^2 - \partial_{11} \phi
\partial_{22} \phi}_{:=g(x)} \Bigr).
\end{align*}

Easy to check that $\hat g(0)=0$ and $(\partial_{\xi} \hat g)(0)=0$.
The latter is due to the fact that
\begin{align*}
\int_{\mathbb R^n} ( (\partial_{12}\phi)^2 - \partial_{11} \phi \partial_{22} \phi ) x_j dx =0, \quad \forall\, 1\le j\le n.
\end{align*}
To derive this, one can use the simple identity (below $\beta=\beta(x)$ denotes a smooth weight function)
\begin{align*}
\int \partial_{11} \phi  \partial_{22} \phi \beta(x) dx &=
\int (\partial_{12} \phi)^2 \beta\notag \\
& \quad\quad+\frac 12\int ( - (\partial_2 \phi)^2 \partial_{11} \beta - (\partial_1 \phi)^2 \partial_{22} \beta
+ 2\partial_1 \phi \partial_2 \phi \partial_{12} \beta) dx.
\end{align*}

On the other hand, by using the identity above, it is not difficult to check that
\begin{align*}
&\int_{\mathbb R^n} g(x) x_1^2 dx =  \int_{\mathbb R^n} (\partial_2\phi)^2
dx,\\
& \int_{\mathbb R^n} g(x) x_2^2 dx =  \int_{\mathbb R^n} (\partial_1\phi)^2
dx, \\
& \int_{\mathbb R^n} g(x) x_1 x_2 dx =-\int_{\mathbb R^n} \partial_1 \phi \partial_2 \phi dx =0.
\end{align*}

Thus near $\xi=0$,
\begin{align*}
\hat g(\xi)= c_1 \xi_1^2 +c_2 \xi_2^2 + O(|\xi|^4),
\end{align*}
where $c_1\ne c_2$. Clearly then $\hat g(\xi) / |\xi|^2$ is not continuous at $\xi=0$.
This immediately implies that $p \notin L^1(\mathbb R^n)$.

\end{proof}

\end{document}